\documentclass[12pt,a4paper,oneside]{amsart}
\usepackage{amsfonts, amsmath, amssymb, amsthm, mathtools}
\usepackage[colorlinks=true,citecolor=blue]{hyperref}
\usepackage[margin=1.4in]{geometry}
\usepackage[T1]{fontenc}
\usepackage{newtxtext}
\usepackage{newtxmath}
\usepackage{graphicx}
\usepackage{verbatim}
\usepackage{tikz}
\usepackage{fix-cm}
\usepackage{tikz-cd}

\newtheorem{theorem}{Theorem}
\newtheorem*{theorem*}{Theorem}
\newtheorem{lemma}[theorem]{Lemma}
\newtheorem*{lemma*}{Lemma}
\newtheorem{prop}[theorem]{Proposition}

\newtheorem{claim}[theorem]{Claim}
\newtheorem{corollary}[theorem]{Corollary}

\newtheorem*{corollary*}{Corollary}

\theoremstyle{definition}

\newtheorem{remark}[theorem]{Remark}
\newtheorem*{remark*}{Remark}

\newtheorem{problem}[theorem]{Problem}

\newtheorem*{problem*}{Problem}

\newtheorem*{conj*}{Conjecture}

\newcommand{\conv}{\mbox{conv}}

\DeclareMathOperator{\bary}{sd}

\newcommand{\cubical}[2]{Q_{#1}^{#2}}

\newcommand{\bry}[1]{\partial #1}
\newcommand{\slice}[2]{S_{#1,#2} } 
\newcommand{\co}{\text{co}N}
\newcommand{\cN}[1]{\text{co}N(#1)}
\newcommand{\faces}[2]{{#1}^{(#2)}}
\newcommand{\inv}[1]{#1^{-1}}

\begin{document} 

\title{Overlap-Helly theorems}

\author{Andreas F. Holmsen}
\author{Alfredo Hubard}

\thanks{The first author was supported by the Institute for Basic Science (IBS-R029-C1)}

\date{\today}

 \address{Andreas F. Holmsen, 
 \hfill \hfill \linebreak 
 Department of Mathematical Sciences,  \hfill \hfill \linebreak
KAIST, 
 Daejeon, South Korea,  \hfill \hfill \linebreak and \hfill \hfill \linebreak
 Discrete Mathematics Group, \hfill\hfill\linebreak
 Institute for Basic Science, \hfill\hfill\linebreak
  Daejeon, South Korea. }
 \email{andreash@kaist.edu}

  \address{Alfredo Hubard, 
 \hfill \hfill \linebreak 
 LIGM,  \hfill \hfill \linebreak
Universite Gustave Eiffel, 
 Champs-sur-Marne, France.  \hfill \hfill }
 \email{alfredo.hubard@univ-eiffel.fr}

\begin{abstract} 
In this paper we introduce a generalization of Helly's theorem closely connected to B\'ar\'any-Gromov overlap theorems (also called selection lemmas). Our main result implies Kalai and Meshulam's colorful Helly theorem for good covers in $\mathbb{R}^d$, as well as Karasev's topological centerpoint theorem. We further investigate the topological fractional Helly theorem from this overlap perspective, and show an overlap theorem for dense complexes (a continuous second selection lemma for tame maps).
\end{abstract}

\maketitle 

\section{Introduction}

\subsection{} An \emph{overlap theorem} is a geometric-topological analogue of the pigeonhole principle. It asserts that any map in a certain class of functions from $X$ to $Y$ must have a large fiber.
The main goal of this paper is to show that several classical Helly-type theorems in combinatorial convexity, as well as their topological analogues, can be viewed as special instances of suitable overlap theorems. A similar point of view taken of Radon--Tverberg type results has received significant attention during the last few decades.

In our results, the target space will be the polyhedron $\|Y\|$ of a (locally finite) simplicial complex $Y$; that is, the underlying topological space of $Y$, consisting of the union of its simplices with the topology induced by the geometric realization.
Our first main result of this type is an overlap generalization of the \emph{colorful Helly theorem} \cite{col-hell}. 

\begin{theorem}\label{thm:convex version}
Let $m >  d\geq 1$ be integers, let $F_1, \dots, F_{d+1}$ be finite families of compact convex sets in $\mathbb{R}^m$, and let $Y$ be a $d$-dimensional simplicial complex. If $C_1\cap \cdots \cap C_{d+1}\neq\emptyset$ for every choice $C_1\in F_1, \dots, C_{d+1}\in F_{d+1}$, then for every continuous map $f \colon \mathbb{R}^m \to \|Y\|$ there exists a point $y\in \|Y\|$ such that $f^{-1}(y)$ intersects every member of one of the families $F_i$.
\end{theorem}

Note that the special case when $Y$ is a triangulation of $\mathbb{R}^d$ and $f$ is a \emph{linear map}, follows from the classical colorful Helly theorem  \cite{col-hell} (discovered by Lov{\'a}sz and independently by B{\'a}r{\'a}ny), since in this case convex sets are mapped to convex sets. Theorem \ref{thm:convex version} thus extends the colorful Helly theorem in two directions: $f$ may be any continuous map, and $Y$ may be any $d$-dimensional simplicial complex. Setting $F_1 = \cdots = F_{d+1}$ gives an overlap-type generalization of Helly's classical theorem on intersections of convex sets. 

\begin{corollary}\label{cor:just helly}
Let $m >  d\geq 1$ be integers, let $F$ be a finite family of compact convex sets in $\mathbb{R}^m$, and let $Y$ be a $d$-dimensional simplicial complex. If every $d+1$ or fewer members of $F$ have nonempty intersection, then for every continuous map $f \colon \mathbb{R}^m \to \|Y\|$ there exists a point $y\in \|Y\|$ such that $f^{-1}(y)$ intersects every member of $F$.
\end{corollary}

One of the common topological generalizations of Helly's theorem replaces the convex sets by a \emph{good cover}; a finite family of compact subsets of some topological space such that the intersection of any subfamily is either empty or contractible. Helly's theorem for good covers in $\mathbb{R}^d$ (discovered by Helly himself \cite{Helly1930}) is a simple consequence of Borsuk's nerve theorem \cite{Borsuk1948}, while the colorful Helly theorem for good covers in $\mathbb{R}^d$ is a substantially deeper result due to Kalai and Meshulam \cite{KM}. The following proposition implies that any Helly-type theorem for good covers is an immediate consequence of its corresponding overlap version.

\begin{prop}\label{prop:good covers}
Let $Y$ be a contractible $d$-dimensional simplicial complex and let $G  =\{S_1, \dots, S_n\}$ be a good cover in $Y$. Suppose $K$ is a simplicial complex on $n$ vertices which is a subcomplex of the nerve of $G$. Then there exists a family of convex sets $F = \{C_1, \dots, C_n\}$ in $\mathbb{R}^{m}$ (for some $m$) and a continuous map $f \colon \mathbb{R}^{m} \to \|Y\|$ such that the following hold.
\begin{enumerate}
\item \label{condition: nerve} The nerve of $F$ equals $K$.
\item \label{condition: containment} For every $i\in [n]$ we have $f(C_i)\subset S_i$.
\end{enumerate}
\end{prop}

Setting $K = [n_1]* \cdots * [n_{d+1}]$ in Proposition \ref{prop:good covers} and applying Theorem \ref{thm:convex version} gives us the colorful Helly theorem for good covers in a contractible $d$-dimensional simplicial complex. In fact, for the case of the colorful Helly theorem, we can further weaken the good cover condition. For a topological space $Y$ we say $\pi_k(Y)$ is \emph{trivial} if $Y$ is nonempty (for $k=-1$), $Y$ is path-connected (for $k=0$), or $\pi_k(Y)=0$ (for $k\geq 1$). The convention for $k=-1$ lets us treat nonemptiness uniformly with the higher connectivity conditions used below.

\begin{corollary}\label{cor:good covers}
Let $Y$ be a $d$-dimensional simplicial complex such that $\pi_d(Y)$ vanishes. Let $F_1$, $\dots$, $F_{d+1}$ be finite families of subcomplexes of $Y$ such that for every nonempty $I\subset [d+1]$ and for every choice $S_i\in F_i$ with $i\in I$, we have $\pi_{d-|I|}(\bigcap_{i\in I}S_i)$ is trivial. Then there is a point $y\in \|Y\|$ that belongs to every member of one of the families $F_i$.  \end{corollary}

Kalai and Meshulam's topological colorful Helly theorem is based on algebraic techniques, and as a payoff their methods extend to \emph{$d$-Leray complexes}, and also allow for replacing the color classes by more general \emph{matroidal colorings}. (For further algebraic generalizations of the colorful Helly theorem, see also Fl{\o}ystad \cite{Floystad2011} and Kalai--Meshulam \cite{KalaiMeshulam2021}.) Our proof of Theorem \ref{thm:convex version} (and Corollary \ref{cor:good covers}) is more elementary than theirs, and is essentially based on Gale's classical HEX Lemma \cite{Gale1979}. It is unclear whether our methods can be extended to general matroidal colorings.

\subsection{}
To further discuss topological overlap theorems we need some notation and definitions. For a simplicial complex $X$, let $\faces{X}{k}$ denote the set of
$k$-dimensional faces. Let $Y$ be a topological space, and let $\mathcal{C}=\mathcal{C}(X,Y)$ be a class of functions from $X$ to $Y$.  For a given $0 < s \leq 1$, we say that $\mathcal{C}$ satisfies the \emph{$s$-overlap property in dimension $k$}, if for every $f \in \mathcal{C}$, there exists $y \in Y$, such that $f^{-1}(y)$ intersects at least an $s$-fraction of the $k$-faces of $X$; that is, \[|\faces{X}{k} \cap f^{-1}(y)| \geq s |\faces{X}{k}|.\]

This terminology allows us to formulate several known overlap theorems. For instance, B{\'a}r{\'a}ny's overlap theorem (also called the first selection lemma \cite[Theorem 9.1.1]{Matousek2002}) asserts that the class of facewise linear maps from the $N$-dimensional simplex, $\Delta^N$, to $\mathbb{R}^d$ satisfies the $(d+1)^{-d}$-overlap property in dimension $d$. This was subsequently improved by Gromov \cite{Gromov2010} who showed that the class of continuous maps from $\Delta^N$ to $\mathbb{R}^d$ satisfies the $\frac{2^d}{(d+1)(d+1)!}$-overlap property in dimension $d$, a better bound than B{\'a}r{\'a}ny's. (We will discuss more overlap theorems later on.) Another result in this direction is 
Karasev's topological generalization of the centerpoint theorem \cite{roman-central}.

\begin{theorem*}[Topological centerpoint theorem]
Let $Y$ be a $d$-dimensional metric space and let $X=\Delta^N$ denote the simplex of dimension $N = (d+1)(k-1)$. The space of continuous functions from $X$ to $Y$ satisfies the $1$-overlap property in dimension $d(k-1)$.
\end{theorem*}
Notice that $d(k-1)=\frac{d}{d+1} N$, so restricted to the class of facewise linear maps into $\mathbb{R}^d$, this reduces to the centerpoint theorem \cite{Rado1946}. Note also that Karasev's theorem, in the case where $Y$ is triangulable, follows from Corollary \ref{cor:just helly} by setting $F$ to be the family consisting of all $d(k-1)$-dimensional faces of $\Delta^{(d+1)(k-1)}$. (Our proof is arguably more elementary than Karasev's proof.)

Now we formulate a special instance of Theorem \ref{thm:convex version}, which is in fact equivalent to it, but which we find to be interesting in its own right. 
Denote by $[1,n]$ the interval of real numbers between $1$ and $n$ (not to be confused with $[n]$ which denotes the integers $\{1, \dots, n\}$). Thus $[1,n]^{d+1}$ is a $(d+1)$-dimensional cube of side length $n-1$, and we let $\cubical{n}{d+1}$ denote the standard cubical subdivision of $[1,n]^{d+1}$.

Next, given $i \in [d+1]$ and $j \in [n]$ define the $j$-th \emph{slice of $\cubical{n}{d+1}$ in direction $i$} to be the subset
\[ \slice{i}{j} = \left\{ \: (x_1, \dots, x_{d+1}) \in [1,n]^{d+1}  :  x_i = j \:\right\}. \]
So each $\slice{i}{j}$ can be identified with a subcomplex of the $d$-skeleton of $\cubical{n}{d+1}$.
The families $F_i = \{\slice{i}{j}  :  j \in [n]\}$, $1\leq i \leq d+1$, serve as a standard example illustrating the optimality of the number of families in the classical colorful Helly theorem (in dimension $d+1$). Indeed, every intersection $\slice{1}{j_1}\cap \cdots \cap \slice{d+1}{j_{d+1}}$ is nonempty, yet the members of each $F_i$ are pairwise disjoint. The following theorem is an immediate consequence of Theorem \ref{thm:convex version}.

\begin{theorem}\label{t:top col helly}
Let $Y$ be a $d$-dimensional simplicial complex. For every continuous map  $f \colon \cubical{n}{d+1} \to \|Y\|$ there exists an $i \in [d+1]$ and a point $y\in \|Y\|$ such that $f^{-1}(y)$ intersects $\slice{i}{j}$ for every $1\leq j\leq n$.
\end{theorem}

We note the Theorem \ref{t:top col helly} is in fact a simple consequence of Lebesgue's covering theorem \cite{lebesgue1911,hurewiczwallman}. However, we will deduce Theorem \ref{t:top col helly} from a suitable version of Gale's HEX Lemma \cite{Gale1979} in Section \ref{sec:sliced-cube}, where we also prove Theorem \ref{thm:convex version}. 

\subsection{}
We believe that essentially every familiar Helly-type theorem admits an overlap generalization. One goal we have not been able to achieve by our methods is the overlap version of the Alon--Kleitman $(p,q)$ theorem \cite{AlonKleitman1992}. 
We say that a family of sets has the \emph{$(p,q)$ property} if every $p$ members of the family contain a subfamily of size $q$ with nonempty intersection.

\begin{problem}\label{con:pq}
For all integers $p\geq q > d$, does there exist an integer $t = t(p,q,d)$ such that the following holds? 

If $F$ is a finite family of convex sets in $\mathbb{R}^n$ with the $(p,q)$ property, then for any continuous map $f: \mathbb{R}^n \to \mathbb{R}^d$ there exists $t$ points $y_1$, $\dots$, $y_t \in \mathbb{R}^d$ such that $\inv{f}(y_1) \cup \cdots \cup \inv{f}(y_t)$ intersects every member of $F$.  
\end{problem}

Just as in Theorem \ref{thm:convex version}, it also makes sense to replace $\mathbb{R}^d$ by any $d$-dimensional simplicial complex $Y$. The same goes for the next two questions as well, even though we only state them for the case $Y = \mathbb{R}^d$.

One of the most important ingredients of Alon and Kleitman's proof of the $(p,q)$ theorem is the \emph{fractional Helly theorem} of Katchalski and Liu \cite{KatchalskiLiu1979} (see also \cite{akmm, holmsen-lee}). We conjecture that the fractional Helly theorem also admits an overlap version. 

\begin{problem}\label{con:FracHelly}
For every integer $d\geq 1$ and $\alpha \in (0,1]$, does there exist a constant $\beta = \beta(\alpha, d) >0$ such that the following holds? 

If $F$ is a finite family of convex sets in $\mathbb{R}^n$ such that at least $\alpha \binom{|F|}{d+1}$ of the $(d+1)$-membered subfamilies have nonempty intersection, then for any continuous map $f \colon \mathbb{R}^n \to \mathbb{R}^d$ there exists a point $y\in \mathbb{R}^d$ such that  $\inv{f}(y)$ intersects at least $\beta|F|$ members of $F$.
\end{problem}

An important result from discrete geometry is the weak $\varepsilon$-net theorem for convex sets, discovered for convex sets in the plane by B{\'a}r{\'a}ny--F{\"u}redi--Lov{\'a}sz \cite{bfl}, and in higher dimensions by Alon--B{\'a}r{\'a}ny--F{\"u}redi--Kleitman \cite{abfk}. Here is the overlap version of the weak $\varepsilon$-net theorem.

\begin{prop}\label{prop:weak epsilon}
For every integer $d\geq 1$ and $\varepsilon \in (0,1)$ there exists an integer $p = p(\varepsilon, d)$ such that the following holds. For every continuous map $f : \Delta^N \to \mathbb{R}^d$ there exists $p$ points $y_1$, $\dots$, $y_p \in \mathbb{R}^d$ such that $\inv{f}(y_1) \cup \cdots \cup \inv{f}(y_p)$ intersects every face of $\Delta^N$ of dimension $\lceil \epsilon N \rceil$.
\end{prop}

The proof of Proposition \ref{prop:weak epsilon} is a straight-forward adaptation of the proof which relies on B{\'a}r{\'a}ny's selection lemma (see e.g. \cite[Theorem 9.2.1]{Matousek2002}), but using Gromov's generalization instead. This is given in Section~\ref{sec:fractional}. Observe that if $K$ is a simplicial complex on $n$ vertices, then the class of continuous functions from $K$ to $\mathbb{R}^d$ satisfy the $s$-overlap property in dimension $\lceil \varepsilon n \rceil$ for $s = \frac{1}{p}$ where $p$ is the value from Proposition \ref{prop:weak epsilon}. Indeed, given a map from $K$ to $\mathbb{R}^d$ extend it arbitrarily to the entire simplex on the vertex set of $K$ and apply Proposition \ref{prop:weak epsilon}. By the pigeon-hole principle one of the points $y_i$ intersects at least a $\frac{1}{p}$-fraction of the $\lceil \varepsilon n \rceil$-faces of $K$.

The so-called second selection lemma \cite{bfl, abfk} (see also \cite[Theorem 9.2.1]{Matousek2002}) asserts that the class of facewise linear functions from a simplicial complex $K$ to $\mathbb{R}^d$ satisfies the $s$-overlap property in dimension $d$, where $s$ depends only on $d$ and the density of the $d$-skeleton of $K$. One possible approach to attacking Problem \ref{con:FracHelly} would be to extend the second selection lemma to all continuous maps. (Note that Pach's overlap theorem is false in the case of $2$-tame maps for any manifold and for $1$-tame maps on the sphere \cite{BaranyMeshulamNevoTancer2018, BukhHubard2020}.)

\begin{problem}\label{con:top2ndsel}
For every $\alpha\in (0,1]$ and integer $d\geq 1$, does there exist a constant $s = s(\alpha, d)$ such that the following holds? 

If $K$ is a simplicial complex on $n$ vertices with at least $\alpha \binom{n}{d+1}$ faces of dimension $d$, then the class of continuous functions from $K$ to $\mathbb{R}^d$ satisfies the $s$-overlap property in dimension $d$.
\end{problem}

In Section \ref{sec:fractional} we show that an affirmative answer to Problem \ref{con:top2ndsel} implies an affirmative answer to Problem \ref{con:FracHelly}. In particular we prove the following.

\begin{prop}\label{prop:2ndsel implies fractional}
Let $K$ be a simplicial complex and suppose that the class of continuous functions from $K$ to $\mathbb{R}^d$ satisfies the $s$-overlap property in dimension $d$. Then, an overlap fractional Helly theorem holds for any family $F$ of convex sets in $\mathbb{R}^m$ whose nerve is isomorphic to $K$. More specifically, for any continuous function $f: \mathbb{R}^m \to \mathbb{R}^d$ there exists a point $y\in \mathbb{R}^d$ such that $f^{-1}(y)$ intersects at least $\beta\cdot |F|$ members of $F$, where $\beta\geq \frac{s}{d+1}$.
\end{prop}

For the remainder of this section we state some special cases of Problem~\ref{con:top2ndsel} that we are able to establish. Firstly, if $\mu_d$ denotes the best constant such that the $(n-1)$-simplex satisfies a $\mu_d$-overlap theorem in dimension $d$, and $a\in (0,1)$, $\alpha>1-a\mu_d$, then an $\alpha$-dense complex $X$ satisfies a $(1-a)\mu_d$-overlap theorem. Indeed we can extend any continous map $X\to \mathbb R^d$ to the whole simplex, apply Gromov's result and then erase the non-faces of $X$. 

A second case is that of random simplicial complexes. We expect that this is known to experts, but we have not found a reference for it, so for completeness we give a proof in Section \ref{sec:fractional}. 

\begin{prop}\label{prop:random} Let \(K\) be the
random \(d\)-dimensional subcomplex of $[n]^{*(d+1)}$ obtained by retaining each \(d\)-face 
independently with probability \(p\), while keeping the full
\((d-1)\)-skeleton. Then with probability tending to one as $n$ tends to infinity, the class of continuous maps from $K$ to $\mathbb R^d$ satisfies the $\mu_d$-overlap property in dimension $d$ with $\mu_d=\frac{1}{2^{d^2+3}d^{d+1}(d+1)!}$. 
\end{prop}

By Proposition \ref{prop:2ndsel implies fractional} this implies an overlap fractional Helly theorem for certain random intersection patterns. 

The second special case is for a restricted class of continuous maps. Let $K$ be a $d$-dimensional simplicial complex and let $M^d$ be a $d$-manifold. Define an $\ell$-tame map  $f \colon K \to M^d$, to be a generic piecewise linear map in which for any pairwise disjoint faces $\sigma_1,\sigma_2, \ldots \sigma_m$ with $\sum_{i=1}^{m}(d-\dim(\sigma_i))=d$, the intersection $\bigcap_{i=1}^{m} f(\sigma_i)$ has at most $\ell$ points.

\begin{theorem}\label{tame} For every $\alpha\in (0,1]$ and integers $d\geq 1$, $\ell\geq 1$ there exists a constant $s_0 = s_0(\alpha,d,\ell)>0$ such that the following holds. If $K$ has $n$ vertices and at least $\alpha \binom{n}{d+1}$ faces of dimension $d$, then the class of $\ell$-tame maps from $K$ to $M^d$ satisfies the $s$-overlap property in dimension $d$ for some $s\geq s_0$.
\end{theorem}

\section{Overlap for the sliced cube}\label{sec:sliced-cube}
A $d$-dimensional simplicial complex is called {\em balanced} if its underlying graph
(i.e., the $1$-skeleton) admits a proper $(d+1)$-coloring, where adjacent vertices are given distinct colors. Note that if $K$ is a simplicial complex, then its barycentric subdivision $\text{sd}(K)$ is balanced; the vertices of $\text{sd}(K)$ are in bijective correspondence with the faces of $K$, and assigning to each vertex of $\text{sd}(K)$ the dimension of its corresponding face in $K$ yields a proper $(d+1)$-coloring. 

\medskip Here is the key lemma needed for the proof of Theorem \ref{t:top col helly}.

\begin{lemma}\label{l:simplicial}
Let $T$ be a triangulation of the $(d+1)$-dimensional cube and let $L$ be a $d$-dimensional balanced simplicial complex. For any simplicial map $f \colon T \to L$ there exists a vertex $v\in L$ such that $f^{-1}(v)$ contains a path that connects disjoint facets of the cube. 
\end{lemma}

We prove Lemma \ref{l:simplicial} using the following well-known variant of Gale's HEX Lemma \cite{Gale1979}. 

\begin{lemma*}[HEX]
Let $T$ be a triangulation of the $(d+1)$-dimensional cube. For any function $h \colon V(T) \to [d+1]$  there exists a path $\gamma$ in the 1-skeleton of $T$ that connects disjoint facets of the cube such that $h(v)$ is constant for every vertex $v$ in $\gamma$.
\end{lemma*}

\begin{proof}[Proof of Lemma \ref{l:simplicial}]
Consider a proper $(d+1)$-coloring,  $\chi : V(L) \to [d+1]$,  of the underlying graph of $L$. 
This induces a function $h \colon V(T) \to [d+1]$ by setting 
\[h(v) = \chi\big( f(v) \big).\]
By the HEX Lemma there exists a path $\gamma$ in $T$ that connects disjoint facets of the cube such that $h(v)$ is constant for every vertex $v$ in $\gamma$, and we claim that also $f(v)$ is constant for every $v$ in $\gamma$. Indeed, any two neighboring vertices $v$ and $w$ in $\gamma$ must either map to the same vertex, or map to an edge of $L$. The latter case is not possible since $L$ is balanced. 
\end{proof}

Before getting to the proof of Theorem \ref{t:top col helly} we make a few more observations. We recall that $\cubical{n}{d+1}$ is a cubical complex, where the vertices are the integer lattice points $[n]^{d+1}$ and the maximal cells are translates of the unit cube $[0,1]^{d+1}$. When we speak of a triangulation of $\cubical{n}{d+1}$ we will always consider a triangulation which refines the natural cubical structure. In this way the sets $\slice{i}{j}$ appear naturally as subcomplexes of any triangulation of $\cubical{n}{d+1}$.

\begin{proof}[Proof of Theorem \ref{t:top col helly}]
We first consider the case of a simplicial map $f \colon T \to L$ where $T$ is a triangulation of $\cubical{n}{d+1}$ and $L$ is a $d$-dimensional balanced simplicial complex.  By Lemma \ref{l:simplicial} there exists a vertex $v\in L$ and a path $\gamma$ in $T$ which connects two disjoint facets of $\cubical{n}{d+1}$ such that $f(w) = v$ for every vertex $w$ in $\gamma$. Suppose the two facets connected by $\gamma$ are orthogonal to the standard unit vector $e_i \in \mathbb{R}^{d+1}$. This means that $\gamma$ intersects the sets $\slice{i}{j}$ for every $j\in [n]$, and therefore $v\in f(\slice{i}{j})$ for every $j\in [n]$.

The case of a general continuous map $f \colon \cubical{n}{d+1} \to \|Y\|$ now follows by standard compactness arguments. Indeed, by compactness, there exists an $\varepsilon > 0$ such that 
\[\bigcap_{j\in J\subset [n]}f(\slice{i}{j}) = \emptyset \implies \bigcap_{j\in J\subset [n]} \big( f(\slice{i}{j})+B_\varepsilon\big) = \emptyset,\] where $(X+B_\varepsilon)$ denotes the Minkowski sum of $X$ and a ball of radius $\varepsilon$ centered at the origin. Let $L_Y$ be a triangulation of $Y$ (which exists by hypothesis). For $m$ sufficiently large, there exists a triangulation $T$ of $\cubical{n}{d+1}$ and a simplicial approximation $\hat{f} \colon T \to \text{sd}^m(L_Y)$ of $f$, such that $\hat{f}(\slice{i}{j}) \subset \big( f(\slice{i}{j}) + B_\varepsilon \big)$ for every $i\in [d+1]$ and $j\in [n]$. Since $\text{sd}^m(L_Y)$ is balanced, the result now follows by applying the simplicial case to the simplicial map $\hat{f}$. 
\end{proof}

\begin{proof}[Proof of Proposition \ref{prop:good covers}]
For a simplicial complex $K$, we let $N(K)$ denote the nerve of the facets of $K$. We denote by $\cN{K}$ the \emph{conerve} of $K$, which is defined as the simplicial complex that has the faces of $K$ as its vertex set, and a subset of vertices is a face of $\cN{K}$ if the corresponding faces of $K$ have nonempty intersection.\footnote{To our knowledge, the term \emph{conerve} was coined by Attila Jung.}

We observe that $N(\cN{K})  = K$. Indeed, the facets of $\cN{K}$ are the stars $\mathrm{st}(v):=\{\tau\in K : v\in \tau\}$ for $v \in \faces{K}{0}$; these are maximal among the collections of faces of $K$ with a common vertex, and every facet of $\cN{K}$ arises this way. Thus the vertices of $N(\co(K))$ are the facets $\mathrm{st}(v)$, which we identify with the vertices $v$ of $K$. A set of facets $\mathrm{st}(v_1),\dots,\mathrm{st}(v_r)$ forms a face of $N(\co(K))$ if and only if they have a common element, that is, a face of $K$ containing $v_1,\dots,v_r$, which happens if and only if $\{v_1,\dots,v_r\}$ is a face of $K$. Hence $N(\co(K))=K$.

We identify the conerve $\cN{K}$ with its geometric realization in $\mathbb{R}^m$, for some sufficiently large $m$.  Let $F = \{C_1, \dots, C_n\}$ be the collection of facets of $\cN{K}$ which is a family of convex sets in $\mathbb{R}^m$, and by the observation above the nerve of $F$ is isomorphic to $K$. It remains to define the map $f$. 

We start by defining the map $f$ on $\cN{K}$, so that the containment condition \eqref{condition: containment} is satisfied. The map can then be extended to all of $\mathbb{R}^m$ using Tietze's Extension Theorem. We define $f$ inductively on the skeleton of $\cN{K}$. To define $f$ on the $0$-skeleton, for every $\sigma\in K$, choose a point $p_\sigma \in \bigcap_{i\in\sigma}C_i$ and define $f(\sigma) = p_\sigma$. Now suppose that $f$ has been defined on the $(k-1)$-skeleton of $\cN{K}$. Let $\tau$ be a $k$-face of $\cN{K}$, that is, $\tau$ consists of faces, $\sigma_1$, $\dots$, $\sigma_{k+1}$, of $K$ where $\sigma = \bigcap_{i=1}^{k+1} \sigma_i$ is a nonempty face of $K$. By hypothesis, the intersection $\bigcap_{i\in \sigma} S_i$ is contractible, and by the inductive assumption, for every proper face $\tau'\subset \tau$ we have $f(\tau') \subset \bigcap_{i\in \sigma} S_i$. Therefore $f$ can be extended on $\tau$, and proceeding in this way, $f$ extends to all of $\cN{K}$.

It remains to show that there is a continuous extension $\hat{f} \colon \mathbb{R}^m \to \|Y\|$ such that $\hat{f}|_K = f$. By \cite[Ch.~III, Cor.~8.4]{Hu1965TheoryOfRetracts}, the locally finite simplicial polyhedron $\|Y\|$ (called a \emph{polytope} in Hu's terminology) is an ANR (Hu's unqualified terms ``AR'' and ``ANR,'' used from \S6 onward, mean ``for the class $\mathcal{M}$ of all metrizable spaces''); by \cite[Ch.~III, Thm.~3.2]{Hu1965TheoryOfRetracts}, it is therefore an ANE for the class $\mathcal{M}$; and since it is contractible, \cite[Ch.~II, Thm.~7.1]{Hu1965TheoryOfRetracts} shows it is an AE for that class---i.e.\ every continuous map from a closed subset of a metrizable space into $\|Y\|$ extends over the whole space.
\end{proof}

\begin{proof}[Proof of Corollary \ref{cor:good covers}]
Without loss of generality we may assume that $|F_i| = n$ for every $i \in [d+1]$ by adding copies of some of the sets if necessary. Now arbitrarily label the members of $F_i$ as 
$C_{i,1}, \dots, C_{i,n}$. We will construct a continuous map $f \colon \cubical{n}{d+1} \to \|Y\|$ such that 
$f(\slice{i}{j}) \subset C_{i,j}$
for every $i\in [d+1]$ and $j\in [n]$. The result then follows from Theorem \ref{t:top col helly} (which is a special case of Theorem \ref{thm:convex version}).

It remains to construct the map $f$. We consider the natural cubical cell structure of $\cubical{n}{d+1}$, define the map on its vertices, and proceed inductively. For a vertex $v = (t_1, \dots, t_{d+1}) \in [n]^{d+1}$ we choose (arbitrarily) a point 
\[p \in \bigcap_{i\in [d+1]} C_{i,t_i}, \]
which is possible since this intersection is nonempty: it is the case $I = [d+1]$ of the hypothesis, where $\pi_{d-|I|} = \pi_{-1}$ is trivial precisely when $\bigcap_{i\in [d+1]} C_{i,t_i}$ is nonempty. We set $f(v) = p$. This defines the map $f$ on the 0-skeleton of $\cubical{n}{d+1}$. 
Now suppose $f$ has been defined on the $k$-skeleton for some $0\leq k < d$ such that 
\[\tau \in \slice{i}{j} \implies f(\tau) \subset C_{i,j}\] for every cell $\tau$ in the $k$-skeleton of $\cubical{n}{d+1}$.
Now consider a $(k+1)$-cell $\sigma \in \cubical{n}{d+1}$. The cell $\sigma$ is contained in the intersection of a unique $(d-k)$-tuple of axis parallel hyperplanes, that is, there is a $(d-k)$-element subset $I\subset [d+1]$ and a set of integers $\{t_i\}_{i\in I}$, with $t_i \in [n]$, such that 
\[\sigma \subset \bigcap_{i\in I} \slice{i}{t_i}.\]
Since the boundary $\bry{\sigma}$ lies in the $k$-skeleton,  we have 
\[f\big( \bry{\sigma} \big) \subset  \bigcap_{i\in I} C_{i,t_i}. \]
By assumption,
$\pi_{k}(\bigcap_{i\in I}C_{i,t_i})$ is trivial, 
and since $\bry{\sigma}$ is homeomorphic to a $k$-sphere, we can extend $f$ on $\sigma$ in such a way  that $f(\sigma) \subset \bigcap_{i\in I} C_{i,t_i}$. Proceeding in this way defines the map on the $d$-skeleton of $\cubical{n}{d+1}$ such that $f(\slice{i}{j})\subset C_{i,j}$ for every $i\in [d+1]$ and $j\in [n]$. Finally, we extend $f$ to the rest of $\cubical{n}{d+1}$ by extending over each individual $(d+1)$-cell. Such a cell is contained in no slice, so this is the case $I = \emptyset$, where $\bigcap_{i\in \emptyset} C_{i,t_i} = \|Y\|$ and there is no further containment requirement to be satisfied; the extension exists because $\pi_d(Y)$ is trivial, which is the $I=\emptyset$ instance of the hypothesis.
\end{proof}

\begin{proof}[Proof of Theorem \ref{thm:convex version}]
Without loss of generality we may assume that $|F_i| = n$ for every $i \in [d+1]$, since otherwise we can add copies of certain bodies. We label the members of $F_i$ arbitrarily as 
$C_{i,1}, \dots, C_{i,n}$.  We can mimic the proof of Corollary \ref{cor:good covers} to construct a map $g \colon \cubical{n}{d+1} \to \mathbb{R}^n$ such that $g(\slice{i}{j})\subset C_{i,j}$ for every $i\in [d+1]$ and $j\in [n]$. The result now follows by applying Theorem \ref{t:top col helly} to the composition $f\circ g$.
\end{proof}

\section{Overlap and the fractional Helly theorem}\label{sec:fractional}

\subsection{} Here we prove Proposition \ref{prop:2ndsel implies fractional}. The first step of the argument needs the following simple lemma.

\begin{lemma}\label{lem:K in F}
Let $F = \{C_1, \dots, C_n\}$ be a family of convex sets in $\mathbb{R}^m$ and let $K$ be the nerve of $F$. There is a continuous map $g: K \to \mathbb{R}^m$ such that for every nonempty face $\sigma\in K$ we have $g(\sigma)\subset \bigcup_{i\in \sigma}C_i$.
\end{lemma}

\begin{proof}
The vertices of the barycentric subdivision $\bary(K)$ are the nonempty faces of $K$. For every nonempty face $\tau\in K$ we choose an arbitrary point $p_{\tau} \in \bigcap_{i\in \tau} C_i$, and define $g$ on the vertex $\tau$ of $\bary(K)$ by $g(\tau) = p_\tau$; the desired map is obtained by taking the affine extension over each simplex of $\bary(K)$.

To verify the containment, fix a nonempty face $\sigma\in K$. In the barycentric subdivision, $\sigma$ is covered by the simplices of $\bary(K)$ corresponding to flags $\sigma_1 \subset \cdots \subset \sigma_k$ of faces of $\sigma$, and the image under $g$ of such a simplex is the convex hull of $p_{\sigma_1}, \dots, p_{\sigma_k}$. Since the flag is nested, $\sigma_1 \subset \sigma_j$ for every $j$, and therefore
\[p_{\sigma_j} \in \bigcap_{i\in \sigma_j} C_i \subset \bigcap_{i\in \sigma_1} C_i.\]
Thus all of the points $p_{\sigma_1}, \dots, p_{\sigma_k}$ lie in the set $\bigcap_{i\in \sigma_1} C_i$, which is convex, being an intersection of convex sets. Hence their convex hull is contained in $\bigcap_{i\in \sigma_1} C_i$ as well, and in particular in $\bigcup_{i\in \sigma}C_i$ since $\sigma_1 \subset \sigma$. As $\sigma$ is covered by these simplices, we conclude that $g(\sigma)\subset \bigcup_{i\in\sigma}C_i$.
\end{proof}

\begin{proof}[Proof of Proposition \ref{prop:2ndsel implies fractional}]
Let $F = \{C_1, \dots, C_n\}$ be a family of convex sets in $\mathbb{R}^m$, let $K$ be the nerve of $F$, and let $f\colon \mathbb{R}^m \to \mathbb{R}^d$ be a continuous map. Under the assumption that $K$ satisfies the $s$-overlap property in dimension $d$, we want to show that there exists a point $y\in \mathbb{R}^d$ such that $\inv{f}(y)$ intersects at least $\frac{s}{d+1} n$ members of $F$.
 
Let $g \colon K \to \mathbb{R}^m$ be the map from Lemma \ref{lem:K in F}. 
Taking the composition, we get a map $h = f \circ g \colon K \to \mathbb{R}^d$, 
for which there is a point $y\in \mathbb{R}^d$ such that $\inv{h}(y)$ intersects at least $s \binom{n}{d+1}$ of the $d$-faces of $K$, and we let $Z \subset \faces{K}{d}$ be the subset of $d$ faces of $K$ that are intersected by $\inv{h}(y)$. By Lemma \ref{lem:K in F}, for every $\sigma\in Z$, the image $g(\sigma)$ is contained in $\bigcup_{i\in \sigma}C_i$, which means that $\inv{f}(y)\cap C_i\neq \emptyset$ for some $i\in \sigma$. Since a vertex of $K$ is contained in at most $\binom{n-1}{d}$ distinct faces in $Z$, it follows that $\inv{f}(y)$ intersects at least
\[\frac{s\binom{n}{d+1}}{\binom{n-1}{d}}\geq s\frac{n}{d+1}\]
distinct members of $F$. 
\end{proof}

\subsection{} Here we prove Proposition \ref{prop:weak epsilon}.
We mimic the proof of the weak $\varepsilon$-net theorem for convex sets which uses B{\'a}r{\'a}ny's selection theorem. This is possible because we have Gromov's generalization which asserts that the class of continuous maps from $\Delta^N$ to $\mathbb{R}^d$ satisfies the $s$-overlap property in dimension $d$, for $s = \frac{2^d}{(d+1)(d+1)!}$. (Importantly, $s$ depends only on $d$.)

\begin{proof}[Proof of Proposition \ref{prop:weak epsilon}]
We proceed by selecting the points greedily. Set $r = \lceil \varepsilon N \rceil$, and for $i\geq 1$, suppose there is an $r$-face $\sigma_i \in \Delta^N$ such that its image $f(\sigma_i)$ is not intersected by previously selected points $y_1, \dots, y_{i-1}$. In particular, an $r$-face $\sigma_1$ exists since we have not selected any points $y_i$ yet. Applying Gromov's overlap theorem to the restriction of $f$ to the face $\sigma_i$, we find a point $y_i\in \mathbb{R}^d$ such that $\inv{f}(y_i)$ intersects at least 
\[s\cdot \binom{r}{d+1} = s\cdot \binom{\lceil \varepsilon N\rceil}{d+1} \geq s\cdot \varepsilon^{d+1}\cdot \binom{N}{d+1}\]
$d$-faces of $\sigma_i$. Thus every time we find a new point $y_i$, the fiber $\inv{f}(y_i)$ intersects at least an $(s\cdot \varepsilon^{d+1})$-fraction of simplices of $\Delta^N$, none of which were intersected by $\inv{f}(y_1) \cup \cdots \cup \inv{f}(y_{i-1})$. Therefore this selection process can be repeated at most $p \leq s^{-1}\cdot \varepsilon^{-(d+1)}$ times. 
\end{proof}

\subsection{}
Here we prove Proposition \ref{prop:random}. The main technical tool is Gromov's overlap theorem, or rather the cosystolic expansion version due to Dotterer--Kaufman--Wagner \cite{dotterer_kaufman_wagner}, combined with coboundary expansion estimates for random complexes due to Wild \cite{wild_thesis} and Dotterrer--Kahle \cite{dotterer_kahle}.

\begin{proof}[Proof of Proposition \ref{prop:random}]
Since $[n]^{*(d+1)}$ is homotopy equivalent to a wedge of $d$-spheres, its cohomology in dimension $<d$ vanishes. Therefore Wild's recursive coboundary expansion bound
\cite[Proposition~8.6]{wild_thesis} is equivalent to a
$L_0$-cofilling inequality in dimension $d$ with $L_0=(3\cdot 2^{d-1}-1)/(d+1)$.

By the Chernoff--Hoeffding estimates of Dotterrer and Kahle \cite{dotterer_kahle},
a sufficiently dense random subcomplex inherits a cofilling inequality with the
constant at most doubled, together with vanishing cohomology. If we choose slack parameter
$\tfrac14$ on their theorem, it suffices that $p\ge 32(d+2)(d+1)L_0\log n/n$. Thus, with probability
tending to one, $Y$ satisfies a $2L_0$-cofilling inequality in dimension $d$ and has
$\widetilde H^{k}(Y;\mathbb F_2)=0$ for $0\le k\le d-1$. In the lower dimensions the
crude bound $d2^{d}$ applies, and $2L_0\le d2^{d}$, so $Y$ satisfies an $L$-cofilling
inequality in all dimensions $1,\dots,d$ with $L:=d2^{d}$. Vanishing cohomology
leaves no nontrivial cocycles below dimension $d$, so the cosystole hypothesis holds
with $\vartheta=1$. Finally, $Y$ is locally $\varepsilon$-sparse with
$\varepsilon=4(d+1)/n$, with probability tending to one.

The topological overlap theorem \cite[Theorem 8]{dotterer_kaufman_wagner}, applied with
$L=d2^{d}$ and $\vartheta=1$, then yields an overlap fraction of at least
\[
   \frac{1}{2(d+1)!\,L^{d}}-O(d\varepsilon)
   =\frac{1}{2^{\,d^{2}+1}d^{d}(d+1)!}-O(d\varepsilon).
\]
Taking $n\ge 32d^{2}(d+1)(d+1)!(d2^{d})^{d}$ gives
$d\varepsilon\le 1/(2^{\,d^{2}+3}d^{\,d+1}(d+1)!)$, which is a $1/4d$ fraction of the
main term; subtracting it leaves an overlap fraction of at least
\[
   \mu_d=\frac{1}{8d\,(d+1)!\,(d2^{d})^{d}}
        =\frac{1}{2^{\,d^{2}+3}\,d^{\,d+1}\,(d+1)!}
\]
(enlarging $n$ by the implicit constant in $O(d\varepsilon)$ if needed). 
\end{proof}

\subsection{} Here we prove Theorem \ref{tame}. We will prove that the class of $\ell$-tame maps from $K$ to $\mathbb{R}^d$ satisfies the $s$-overlap property in dimension $d$. This means the class of piecewise linear maps from $K$ to $\mathbb{R}^d$ such that for any family of pairwise disjoint faces $\sigma_1,\sigma_2, \ldots \sigma_m\in X$, where $\sum_{i=1}^m (d- \dim(\sigma_i))=d$, we have 
\[\left| \bigcap_{i=1}^m f(\sigma_i) \right|\le \ell.\]

The proof is similar to the proof for face-wise affine maps from \cite{abfk}. We will need the following variant of the colored Tverberg theorem of Zivaljevi\'c and Vre\'cica \cite{zi-vre92}. In the statement and its proof, let $d>0$, let $p$ be a prime, and set $T = [2p-1]^{*(d+1)}$.

\begin{theorem}\label{col_tv} For any continuous map $f\colon T \to \mathbb R^d$, there exists a family of pairwise disjoint faces $\sigma_1,\sigma_2, \ldots \sigma_p \in T$, with $\sum_{i=1}^p |\sigma_i|=(d+1)(p-1)+1$ such that $\bigcap_{i=1}^p f(\sigma_i)\neq \emptyset$.
 \end{theorem}

\begin{remark} \label{rem: d simplices} If $(d+1)(p-1)+1$ points are partitioned into a set $S$ of $p$ simplices, then $\sum_{\sigma \in S} (d-\dim(\sigma))=d$. Furthermore, if $p>d$ there exists some subset $S'\subset S$ with at most $d$ simplices for which $\sum_{\sigma \in S'} (d-\dim(\sigma))=d$, which implies that the intersection of their images under a generic piecewise linear map must be a $0$-dimensional set (possibly empty).
\end{remark}

The proof is essentially the same as the original proof of Zivaljevi\'c and Vre\'cica \cite{zi-vre92} (and \cite{blvz94}). We refer the reader to \cite{matousek2003using} for more details and provide a quick sketch of the proof here. We begin with a proof of the classical colorful Tverberg theorem and explain at the end what needs to be modified to obtain the condition on the number of vertices.

 \begin{proof}[Proof of Theorem \ref{col_tv}]
The continuous map $f$ induces a continuous map from the $p$-fold 2-wise deleted join of $T$ to the $p$-wise join of $\mathbb R^d$. The abelian group with $p$ elements $\mathbb Z_p$ acts freely permuting the factors of this deleted join, and acts by linear isometries permuting the factors of $\mathbb R^d$ in the $p$-fold join. Therefore the induced map
$$f^{*p}\colon T^{*p}_{\Delta(2)} \to \mathbb R^{(d+1)p-1},$$
is $\mathbb Z_p$ equivariant. 
If $f$ doesn't have a $p$-multiple point, then the image of $f^{*p}$ avoids the $d$-dimensional diagonal in $(\mathbb {R}^d)^{*p}=\mathbb R^{(d+1)p-1}$. So we can $\mathbb Z_p$-equivariantly deformation retract $f^{*p}$, to a $\mathbb Z_p$-equivariant map $$g\colon T^{*p}_{\Delta(2)} \to \mathbb S^{(d+1)(p-1)-1}.$$
To show that there does not exist such an equivariant map, a lower bound on the connectedness of $T^{*p}_{\Delta(2)}$ suffices. For this we use the fact that join and deleted join commute, and that if $X$ is $k$-connected and $Y$ is $l$-connected, their join is $(k+l+2)$-connected. So to understand the connectivity of our test space it suffices to know the connectivity of the (so called) \emph{chessboard complex} $[2p-1]^{*p}_{\Delta(2)}$, which is $(p-2)$-connected, as shown in \cite{blvz94}. From this and the join connectivity formula, applied to the $(d+1)$-fold join $T^{*p}_{\Delta(2)}=([2p-1]^{*p}_{\Delta(2)})^{*(d+1)}$, we obtain that $T^{*p}_{\Delta(2)}$ is $((d+1)p-2)$-connected. 

Dold's theorem (see e.g. \cite[Theorem 6.2.6]{matousek2003using}) implies that the equivariant map $g$ does not exist, and we conclude that there exist pairwise disjoint faces $\sigma_1,\sigma_2, \ldots \sigma_p \in T$ such that $\bigcap_{i=1}^p f(\sigma_i)\neq \emptyset$. 

So far we have just rephrased the classical proof of the colored Tverberg theorem. It remains to show the dimension requirement. We use the general fact that a simplicial complex $L$ is $\ell$-connected if and only if its $(\ell+1)$-skeleton is $\ell$-connected. Now observe that the $(d+1)(p-1)$-skeleton of $T^{*p}_{\Delta(2)}$ corresponds precisely to $p$-tuples of disjoint simplices on $(d+1)(p-1)+1$ vertices. This skeleton is also $\big( (d+1)(p-1)-1 \big)$-connected, and so the proof above goes through as before.  
 \end{proof}

We also need the following version of the Erd\H{o}s-Simonovits theorem (see e.g. \cite[Theorem 9.2.2]{Matousek2002}), which we state here in the language of simplicial complexes for convenience

\begin{theorem}[Erd\H{o}s-Simonovits]\label{ES}
For all integers $k$ and $t$, and positive real number $\alpha \in (0,1]$ there exists a $\delta>0$ such that the following holds. If $K$ is a simplicial complex with at least $\alpha n^k$ faces of dimension $(k-1)$, then $K$ contains at least $\delta n^{kt}$ distinct copies of $[t]^{*k}$.
\end{theorem}

\begin{proof}[Proof of Theorem \ref{tame}] We apply Theorem \ref{ES} with $k=d+1$, and choose $t$ to be the smallest number such that we can apply Theorem \ref{col_tv} to $[t]^{*(d+1)}$ in order to find $d+1$ pairwise disjoint simplices $\sigma_0$, $\sigma_1$, $\ldots$,  $\sigma_d$ with $\bigcap_{i=0}^{d} f(\sigma_i)\neq \emptyset$. 
For any $r\geq 2$ there is a prime between $r$ and $2r$, and so Theorem \ref{col_tv} implies that $t\leq 4d-1$.

Applying Theorem \ref{ES} yields $\delta n^{t(d+1)}$ distinct copies of $[t]^{*(d+1)}$. For a fixed copy of $[t]^{*(d+1)}$ we find $(d+1)$ pairwise disjoint simplices $\sigma_0$, $\dots$, $\sigma_{d}$ whose images under $f$ have a point in common. 
As noted in Remark \ref{rem: d simplices} we may assume that $\bigcap_{i=1}^{d} f(\sigma_i)$ is a 0-dimensional set and so by the $\ell$-tameness consists of at most $\ell$ points,  we call these points \emph{crossing points} from now on. At least one of the crossing points is contained in the remaining full dimensional simplex $\sigma_0$, and by the genericity of the map $f$, any such crossing point is in the interior of $\sigma_0$. 

Now we use a double counting argument, counting incidences between crossing points and the full-dimensional simplices whose interior contains them.

Each time a $d$-tuple of (vertex) disjoint simplices gives rise to a set of crossing points, the total number of vertices equals $(d+1)(d-1)+1 = d^2$. Therefore there are at most a total of $\kappa \, n^{d^2}$ possible crossing points, for some positive $\kappa$ that depends on $\ell$ and $d$. The additional dependency on $d$ accounts for the number of distinct partitions of $d^2$ vertices that could give rise to a set of crossing points.

For each of the $\delta n^{t(d+1)}$ distinct copies of $[t]^{*(d+1)}$, by our choice of $t$ we obtain at least one $d$-dimensional simplex that contains a crossing point in its interior (this is the content of Remark \ref{rem: d simplices}, together with the $\ell$-tameness condition). This gives a lower bound on the number of (copy, full-dimensional simplex, crossing point) incidences.

Conversely, we count the same incidences by first fixing a full-dimensional simplex together with the at most $d$ simplices that generate the crossing, and then completing this configuration to a copy of $[t]^{*(d+1)}$. The full-dimensional simplex accounts for $d+1$ vertices, and as argued above an additional $d^2$ vertices generate the crossings. Therefore such a completion can be done in at most $b\, n^{t(d+1)-(d^2+d+1)}$ ways, where $b$ is a constant depending on $d$ which accounts for the different ways that the vertices can be split into $d+1$ parts to form $[t]^{*(d+1)}$.
So the average full-dimensional simplex contains at least
\[\frac{\delta n^{t(d+1)}}{b\, n^{t(d+1)-(d^2+d+1)}}=\frac{\delta}{b}  n^{d^2+d+1}\]
crossing points. Dividing by the total number $\kappa\,n^{d^2}$ of crossing points, the pigeonhole principle yields a crossing point $y$ that is contained in the interior of at least 
\[\frac{\delta}{ \kappa\, b}\, n^{d+1}\]
full-dimensional simplices. Each such full-dimensional simplex $\sigma$ satisfies $y \in f(\sigma)$, that is, $\inv{f}(y)\cap \sigma \neq \emptyset$. Since $K$ has at most $\binom{n}{d+1} = O(n^{d+1})$ faces of dimension $d$, the fiber $\inv{f}(y)$ meets a positive fraction of them, which is the desired $s$-overlap property. 
\end{proof}

Using a more quantitatively precise version of the Erd{\H o}s--Simonovits theorem it is possible to obtain bounds on the dependency on the $s$-overlap property. In particular we can get a lower bound on $s$ of $\Omega\!\left(\alpha^{(4d-1)^{d+1}}\right)$. We omit the details.

We only gave a proof of Theorem \ref{tame} for maps into $\mathbb{R}^d$. Everything generalizes readily to a $d$-manifold $M^d$ using the colorful Tverberg theorem for maps to manifolds by Blagojevi\'c, Matschke and Ziegler \cite{BlagojevicMatschkeZiegler2011}. We omit the details.

\section{Concluding remarks}
While the previous approaches towards Helly type theorems used very different strategies in the $d$-Leray case and in the bounded Betti number case, approaching topological Helly theorems via overlap theorems seems relevant for both families of generalizations.

Kalai and Meshulam's topological colorful Helly theorem is more general than ours in two ways, one aspect is matroidal, it would be interesting if a generalization of our theorem could cover that aspect. The second aspect is the assumption on the ambient space; their result holds for any $d$-Leray complex. We now explain what this means and how our approach could be extended to cover this generalization.

A simplicial complex $K$ is \emph{$d$-Leray} if $\tilde H_i(K') = 0$ for every $i\geq d$ and every induced subcomplex $K'\subset K$. 

\begin{claim} Let $Y$ be a $d$-dimensional simplicial complex with vanishing homology in dimension $d$. If $F$ is a good cover in $Y$, then the nerve of $F$ is $d$-Leray.
\end{claim} 

\begin{proof}
For a subset of vertices $S$, put \(X_S := \bigcup_{i\in S} X_i\), and
\(
N[S] = N\big(\textstyle\bigcup_{i \in S}X_i\big).
\)
By the Nerve Theorem,
\(
N[S] \;\simeq\; X_S.
\)
Now $X_S$ is a subcomplex of $X$, so $\dim X_S \le d$ and therefore $H_i(X_S) = 0$ for all $i > d$.
Since $X$ has no $(d{+}1)$-simplices
\(
H_d(X) = Z_d(X) ,
\)
and by hypothesis $H_d(X)=0$. Since $X_S$ is a subcomplex of $X$, $Z_d(X_S) = 0$ and therefore $H_d(X_S)=0$.
Combining the above, $\tilde H_i(N[S]) \cong \tilde H_i(X_S) = 0$ for every
$i \ge d$ and every $S \subseteq [n]$. 
\end{proof}

For $d=1$, the converse claim also holds, that is, $1$-Leray complexes are clique complexes of chordal graphs, which can be represented as intersection graphs of subtrees of a tree. A possible generalization would be the following.

\begin{problem} Let $K$ be a $d$-Leray complex. Does there exist a $d$-dimensional simplicial complex $Y$ with vanishing homology in dimension $d$, and good cover $F$ in $Y$ such that the nerve of $F$ is isomorphic to $K$?
\end{problem}

Here we describe a different approach to establishing the overlap versions of Helly type theorems. Consider a family $F$ of convex sets in $\mathbb{R}^m$ and let $K$ be the nerve of $F$. For a continuous map $f: \mathbb{R}^m \to \mathbb{R}^d$ let $M$ be the nerve of $f(F)$. Clearly $K$ is a subcomplex of $M$. 

\begin{problem}\label{con:Leray Sandwich}
In the setting described above, does there exist a $d$-Leray complex $L$ such that $K\subset L \subset M$?
\end{problem}

An affirmative answer to Problem \ref{con:Leray Sandwich} would imply that any overlap Helly theorem is a direct consequence of the corresponding Helly theorem for $d$-Leray complexes. 

In a subsequent paper we provide a positive answer to Problem \ref{con:top2ndsel}) and extract some consequences. 

\section{Acknowledgments}
We thank Roman Karasev for comments on a previous version of this paper.

\bibliographystyle{plain}

\bibliography{references}

\end{document}